\documentclass[oneside,openany,a4paper]{amsart}
\usepackage[a4paper, margin=1in]{geometry}
\usepackage{amssymb}
\usepackage[centertags]{amsmath}
\usepackage{amsthm}
\usepackage{amsfonts}
\usepackage{eucal}
\usepackage{mathrsfs}

\usepackage[all,cmtip]{xy}
\usepackage{graphicx, eepic}

\usepackage{lmodern}
\usepackage[utf8]{inputenc}
\usepackage[T1]{fontenc}
\usepackage{xcolor}

\usepackage{bbm}
\usepackage{pdfpages}
\usepackage{titletoc}
\usepackage{booktabs}
\usepackage[shortlabels]{enumitem}
\usepackage{mathtools}
\usepackage{thmtools}

\usepackage{epstopdf}
\usepackage{epsfig}
\usepackage{microtype}

\usepackage{hyperref}
\hypersetup{colorlinks}
\hypersetup{
    bookmarks=true,         
    unicode=false,          
    pdftoolbar=true,        
    pdfmenubar=true,        
    pdffitwindow=false,     
    pdfstartview={FitH},    
    pdftitle={My title},    
    pdfauthor={Author},     
    pdfsubject={Subject},   
    pdfcreator={Creator},   
    pdfproducer={Producer}, 
    pdfkeywords={keyword1} {key2} {key3}, 
    pdfnewwindow=true,      
    colorlinks=true,       
    linkcolor=red,          
    citecolor=magenta,        
    filecolor=violet,      
    urlcolor=blue      
}
\usepackage{multirow, url}
\usepackage{textcomp}
\DeclareMathAlphabet{\cmcal}{OMS}{cmsy}{m}{n}
\newtheoremstyle{thm}
  {3pt}
  {3pt}
  {\em}
  {0pt}
  {\bfseries}
  {}
  {5pt}
  {}
\newtheoremstyle{rem}
  {3pt}
  {3pt}
  {}
  {0pt}
  {\bfseries}
  {.}
  {5pt}
  {}

\newtheorem{thm}{Theorem}[section]
\newtheorem{cor}[thm]{Corollary}
\newtheorem{lem}[thm]{Lemma}
\newtheorem{prop}[thm]{Proposition}

\newtheorem{conj}[thm]{Conjecture}

\theoremstyle{definition}

\theoremstyle{rem}
\newtheorem{rem}[thm]{{Remark}}

\numberwithin{equation}{section} \numberwithin{table}{section}
	
\newtheorem*{thm*}{Theorem}
\newtheorem*{rem*}{Remark}
\newtheorem*{rems*}{Remarks}
\newtheorem*{exam*}{Example}
\newtheorem*{exams*}{Examples}

\makeatletter
\newcommand{\neutralize}[1]{\expandafter\let\csname c@#1\endcsname\count@}
\makeatother

\newcommand{\Q}{{\mathbb{Q}}}
\newcommand{\R}{{\mathbb{R}}}

\newcommand{\bU}{{\mathbb{U}}}

\newcommand{\Z}{{\mathbb{Z}}}

\newcommand{\cF}{{\cmcal{F}}}
\newcommand{\cG}{{\cmcal{G}}}

\newcommand{\cO}{{\cmcal{O}}}
\newcommand{\cP}{{\cmcal{P}}}
\newcommand{\cQ}{{\cmcal{Q}}}

\newcommand{\cT}{{\cmcal{T}}}

\newcommand{\rH}{\mathrm{H}}

\def\n2Z{\frac{1}{n^2}\Z/\Z}
\def\a{\alpha}
\def\b{\beta}

\def\d{\delta}

\def\z{\zeta}

\newcommand{\Gm}{{\mathbb{G}}_{m}}

\newcommand{\zmod}[1]{{\Z/{#1}\Z}}
\newcommand{\modz}[1]{{{\frac{1}{#1}\Z}/\Z}}

\newcommand{\ra}{\rightarrow}

\newcommand{\arinj}{\ar@{^(->}}
\newcommand{\arsurj}{\ar@{->>}}
\newcommand{\arsub}{\ar@{}[r]|-*[@]{\subset}}
\newcommand{\arsup}{\ar@{}[r]|-*[@]{\supset}}
\newcommand{\arcap}{\ar@{}[d]|-*[@]{\subset}}
\newcommand{\arcup}{\ar@{}[u]|-*[@]{\subset}}
\newcommand{\arin}{\ar@{}[u]|-*[@]{\in}}

\renewcommand{\Im}{\operatorname{Im}}
\newcommand{\Ker}{\operatorname{Ker}}

\newcommand{\Ext}{{\operatorname{Ext}}}
\newcommand{\Hom}{{\operatorname{Hom}}}

\newcommand{\Pic}{{\operatorname{Pic}}}

\newcommand{\Spec}{{\operatorname{Spec}}}

\newcommand{\Inv}{{\operatorname{Inv}}}

\renewcommand{~}{\hspace{0.4mm}}
\newcommand{\ts}{{\textsection}}

\newcommand{\bd}{\boldsymbol}

\newcommand{\wg}{\wedge}
\def\tkf{\widetilde{k_n(f)}}
\mathchardef\hyp="2D

\newcommand{\xyv}[1]{\xymatrixrowsep{#1 pc}}
\newcommand{\xyh}[1]{\xymatrixcolsep{#1 pc}}

\newcommand{\sesdot}[3]{
  \xymatrix{
0 \ar[r] & #1 \ar[r] & #2 \ar[r] & #3 \ar[r] & 0.}}

\newcommand{\pair}[2]{\langle #1, \,#2 \rangle}

\newcommand{\legendre}[2]{\genfrac{(}{)}{}{}{#1}{#2}}

\newcommand{\vio}[1]{{\color{black} #1}}
\def\<{\langle }
\def\>{\rangle}
\def\S{\Sigma}
\def\bF{\bar{F}}
\def\nZ{\frac{1}{n}\mathbb{Z}/\mathbb{Z}}
\def\d{\delta}
\def\ms{\medskip}
\def\lk{\ell k}
\def\invlim{\varprojlim}

\def\znn{\zmod n \times_S \mu_{n^2}}
\def\bU{\partial U}

\def\Cl{{ {Cl}}}
\def\calL{{\cmcal {L}}}
\renewcommand{\O}{{\cmcal O}}

\begin{document}                                                                          

\title{Abelian arithmetic Chern-Simons theory and arithmetic linking numbers}

\author{Hee-Joong Chung}
\address{ H.J.C.: Department of Physics, Pohang University of Science and Technology, 77 Cheongam-ro, Nam-gu, Pohang, Gyeongbuk, Republic of Korea 37673, 
and
Korea Institute for Advanced Study, 85 Hoegiro, Dongdaemun-gu, Seoul 02455, Republic of Korea
}
\author{Dohyeong Kim}
\address{D.K.: Department of
Mathematics, University of Michigan,
2074 East Hall,
530 Church Street,
Ann Arbor, MI 48109-1043 , U.S.A. }
\author{Minhyong Kim}
\address{M.K.: Mathematical Institute, University of Oxford,
Woodstock Road, Oxford OX2 6GG, U.K., and Korea Institute for Advanced Study,
85 Hoegiro, Dongdaemun-gu,
Seoul 02455,
Republic of Korea }
\author{George Pappas}
\address{G.P.: Department of Mathematics,
Michigan State University,
East Lansing, MI 48824, U.S.A.}
\author{Jeehoon Park}
\address{J.P.: Department of Mathematics, 
Pohang University of Science and Technology,
77 Cheongam-ro, Nam-gu, Pohang, Gyeongbuk, Republic of Korea 37673 }
\author{Hwajong Yoo}
\address{H.Y.: IBS Center for Geometry and Physics,
Mathematical Science Building, Room 108,
Pohang University of Science and Technology,
77 Cheongam-ro, Nam-gu, Pohang, Gyeongbuk, Republic of Korea 37673 }
\subjclass[2000]{Primary 11R04, 11R23, 11R29 ; Secondary 81T45}                                                                        
\maketitle

\begin{abstract}
Following the method of Seifert surfaces in knot theory, we define  arithmetic linking numbers and height pairings of ideals using arithmetic duality theorems, and compute them in terms of $n$-th power residue symbols. This formalism leads to a precise  arithmetic analogue of a `path-integral formula' for linking numbers.\end{abstract}


\section{Introduction}\label{intro}
Let  $M$ be an oriented three-manifold  without boundary and $\a_1$ and $\a_2$  two knots that are homologically equivalent to zero in it. One way of computing the linking number of $\a_1$ and $\a_2$ uses the formula
$$\lk (\a_1, \a_2)=\<\S_{\a_1 },\a_2 \>,$$
where $\S_{\a_1}$ is a Seifert surface for $\a_1$ transversal to $\a_2$ and $\<\S_{\a_1 },\a_2 \>$ is the oriented intersection number. It is also suggestive to write this equality as
$$\lk (\a_1,\a_2)=\<d^{-1}\a_1, \a_2\>,$$
$d$ denoting the exterior derivative of currents. The pairing on the right is independent of the choice of (smooth, transversal) inverse image: because de Rham cohomology computed by forms and currents is the same, the ambiguity can be represented by closed 1-forms, which then integrate to zero on $\a_2$, since the latter is assumed to be homologically equivalent to zero.
\ms

We can also define a pairing between two 1-forms $A_1$ and $A_2$ by
$$(A_1, A_2):=\pair{A_1}{dA_2}:=\int_M A_1 \wg dA_2.$$
Since
$$d(A_1\wg A_2)=dA_1\wg A_2-A_1\wg dA_2,$$
we see right away that the pairing is symmetric by Stokes' theorem.

According to \cite{AK}, the \textit{Chern-Simons action}
\[
(A,A)=\int_M A \wg dA
\]
for a 1-form $A$ is related to the \textit{helicity} of a magnetic field. Indeed, if $M$ is a space-like slice of the spacetime $M\times \R$ and $A$ the electromagnetic potential, we have the equality

\[
\int_M A \wg dA=\int_M \Phi \cdot B ~d{\text{vol}},
\]
where $B$ is the magnetic field and $\Phi$ the magnetic vector potential.
\ms

Here is an aside about the meaning of the integral $\int_MA \wg dA$ as `helicity'. The choice of a volume form $d\text{vol}$ on $M$ determines an isomorphism  $V\mapsto i_V d\text{vol}$ from vector fields to 2-forms. The vector field $V$ corresponding to $dA$ will generate a flow so that we can consider the trajectory $\ell_p$ that starts from any given point $p$. Arnold and Khesin  \cite{AK} define an asymptotic linking number
$\lk(\ell_p,\ell_q)$ and prove a formula of the form
\[
\int_M A \wg dA = \int_{M} d^{-1}(i_Vd\text{vol})\wedge i_Vd\text{vol}=c\int_{M \times M} \lk(\ell_p,\ell_q) ~d\text{vol}_p ~d\text{vol}_q.
\]
That is, the helicity is an average asymptotic linking number between pairs of magnetic flows starting from two points in $M$.

\ms

Following Polyakov \cite{Pol} and Schwartz \cite{Sch}, they also discuss the formal `Gaussian' path integral 
\[
\int \exp{(-\pi \pair{A}{dA})}~DA = \det(*d)^{-\frac{1}{2}},
\]
where $*:\Omega^2_M\ra \Omega^1_M$ is the Hodge star operator with respect to a metric and the determinant is regularised\footnote{In this and the next formula, we will be somewhat vague with the precise definitions and computations, since we will not be using them in this paper except as inspiration. In particular, \cite{Sch} gives a careful discussion of the metric dependence and the possibility that $d$ has non-trivial kernel. Also, we have normalised the constants slightly differently.} \cite[p. 186]{AK}.
Adding a linear term pairing the forms with homologically trivial currents $\xi_i$,  we get (again formally)
\[
\int \exp {\left(-\pi\pair{A}{dA} +2\pi i \sum_{i} \pair {A}{\xi_i} \right)} ~DA
=  \det (*d)^{-\frac{1}{2}} \cdot \exp{ \left(-\pi \sum_{i, j} \pair {d^{-1}\xi_i}{\xi_j}\right)}.
\]
The pairings between currents on the right side are likely to be problematic in general. However, the case of interest is when the $\xi_i$ are (oriented) knots and the pairing with $A$ denotes an integral. The operator $d$ acts on  currents in a way compatible with boundary maps of singular chains. That is, if  $L, N$ are chains with
$\partial N=L$ and  $[L]$ and $[N]$ are the corresponding currents, then
$d[N]=[\partial N]=[L].$ 
Hence, if $\xi_i$ is a current corresponding to a homologically trivial knot, then $d^{-1}\xi_i$ will include a two-chain with boundary equal to $\xi_i$.  Thus,
 each term $\pair {d^{-1}\xi_i}{\xi_j}=\lk(\xi_i, \xi_j)$ will be a linking number. The integral is thereby viewed as a correlation between the `Wilson loop functionals'
$$A\mapsto \exp {\left( 2\pi i  \pair {A}{\xi_i} \right)}$$
associated to knots $\xi_i$ with respect to a Chern-Simons measure
$$\exp {\left(-\pi(A,A)\right)} ~DA.$$ In any case,
 the Gaussian integral with linear term provides  one elementary explanation of how linking numbers come up in Chern-Simons theory.
 \ms

The goal of this paper is to present some preliminary investigations on arithmetic analogues of the preceding discussion. That is, when $X=\Spec(\cO_F)$ for a totally imaginary number field $F$ that contains the group $\mu_{n^2}$ of $n^2$-th roots of unity, we use arithmetic duality theorems to define a two term complex
$$d: H^1(X, \zmod n)\ra \Ext^2_X(\zmod n, \Gm)$$
as a mod $n$ arithmetic analogue of the map $d:\Omega^1_M\ra \Omega^2_M$. The Ext group is isomorphic to $Cl(F)/n$, the ideal class group of $F$ mod $n$. Thus, every ideal $I$ has a mod $n$ class $$[I]_n\in  \Ext^2_X(\zmod n, \Gm),$$ and we define $I$ to be \textit{$n$-homologically trivial} if this class is in the image of $d$. On the other hand, there is a duality pairing
$$\<\cdot, \cdot \>: H^1(X, \zmod n)\times \Ext^2_X(\zmod n, \Gm)\ra \nZ, $$
and we define \textit{the arithmetic linking number} of two prime ideals $\cP$ and $\cQ$ that are $n$-homologically trivial by
$$\lk_n(\cP, \cQ):=\< d^{-1}\cP, \cQ\>.$$
Of course one needs to check that this is well-defined and symmetric. 
We verify this in Section \ref{section2}. In Section \ref{section3}, we generalise the definition to arithmetic linking numbers on $X_S:=\Spec(\cO_F[1/S])$ for a finite set of primes $S$. We will see (Corollary \ref{cor:3.11})  that this linking number can be computed in terms of $n$-th power residue symbols in a manner similar to Morishita's treatment in \cite{Mor}. (However, we do not carry out a direct comparison.) This pairing can be defined also for non-prime ideals, in which case we call it the {\em arithmetic  mod $n$ height pairing}, denoted by $ht_n(I,J)$.

Parallel to the pairing on 1-forms, we also define a pairing
$$(\cdot, \cdot): H^1(X,\zmod n)\times  H^1(X,\zmod n)\ra \nZ$$
as
$$(A,B)=\<A, d B\>$$
and
in such a way that $(A,A)$ is  the abelian arithmetic Chern-Simons function defined in \cite{Kim, CKKPY}. 

It is then pleasant to note  a precise analogue of the Gaussian path integral in this arithmetic setting.
\begin{thm} \label{thm:maintheorem}
Let $p$ be an odd prime,  $a=\dim H^1(X, \zmod p)$, $b=\dim \Ker(d)$, and $\{\xi_j\}$ a finite set of homologically $p$-trivial ideals. Denote by $\bar{d}$ the induced isomorphism
$$\bar{d}: H^1(X, \zmod p)/\Ker(d) \overset{\sim}{\ra} \Im(d).$$
Then
\begin{align*}
& \sum_{\rho\in H^1(X, ~\zmod p)} \exp[2\pi i ((\rho, \rho )+\sum_j\pair{\rho}{[\xi_j]_p})] \\
& =p^{(a+b)/2}\legendre{\det (\bar{d})}{p}i^{[\frac{(a-b)(p-1)^2}{4}]} \exp \left[-2\pi i \left( \frac{1}{4}\sum_{i, ~j}ht_p(\xi_i, \xi_j) \right) \right].
\end{align*}
\end{thm}
The determinant requires some commentary. The map $\bar{d}$ goes from
$H^1(X, \zmod p)/\Ker (d)$ to its dual, since $\Ker (d)$ is the exact annihilator of $\Im(d)$. It is an easy exercise to check that the determinant is then well-defined modulo squares in $\zmod p$. (It is just the discriminant of the corresponding quadratic form.) Hence, its Legendre symbol is well-defined. This formula is essentially a formal consequence of the definitions. However, it does give indication that some notion of `quantisation' for  arithmetic Chern-Simons theory might not be entirely empty and, furthermore, provide new interpretations of basic arithmetic invariants.

\ms
In Section \ref{section4}, following up on the ideas of \cite{{BCGKPT}}, we will also show how to realize the arithmetic linking pairing in the compact case by a simple construction that only involves Artin reciprocity and the `class invariant homomorphism',
which gives a measure of the Galois structure of unramified Galois extensions.
More precisely, we show that under the class field theory isomorphism $(\Cl(F)/n)^\vee\simeq H^1(X,\Z/n\Z)$ the map
\[
d: H^1(X,\Z/n\Z)\to \Ext^2_X(\zmod n, \Gm)\simeq
H^1(X,\Z/n\Z)^\vee
\]
giving $(\cdot,\cdot)$ is identified with the class invariant homomorphism
\[
(\Cl(F)/n)^\vee=\Hom(\Cl(F), \Z/{n\Z})\to \Cl(F)/n.
\]
By definition, this sends the Artin map of a $\Z/n\Z$-unramified extension $K/F$ to the class
of the (locally free) $\O_F$-submodule of $K$ consisting of $v\in K$ such that $a(v)=\zeta^a v$. 
Regarding Chern-Simons functionals, the first computation in terms of the Artin map was in \cite{{BCGKPT}}. Martin Taylor observed a relation to the class invariant homomorphism when $n=2$, while Romyar Sharifi pointed out a connection to Bockstein maps.
\ms

As mentioned already, many of the ideas of the current paper were explored in various forms and in considerable depth by \cite{Mor}. What we view as the main contribution here, as in \cite{Kim, CKKPY}, is an attempt to move beyond analogies to a precise correspondence of constructions and techniques used in topology (especially the ideas inspired by topological quantum field theory) and in arithmetic geometry. What is achieved is obviously modest. But we hope it is suggestive.

\section{Arithmetic linking numbers in the compact case: proof of Theorem \ref{thm:maintheorem}}\label{section2}
Let $F$ be a totally imaginary algebraic number field with ring of integers $\cO_F$  such that $\mu_{n^2}\subset F$, and let $X=\Spec(\cO_F)$. We fix a trivialisation of the $n$-th roots
$$\z: \zmod n\simeq \mu_n.$$
We have various isomorphisms
$$\z_* : H^i(X,\zmod n)\simeq H^i(X, \mu_n);$$
$$\z^* : \Ext^i_X(\zmod n, \Gm)\simeq \Ext_X^i(\mu_n, \Gm).$$
Let $\pi:=\pi_1(X, b)$, where $b:\Spec(\bF)\ra\Spec(\cO_F)$ is the geometric point coming from an algebraic closure $\bF$ of $F$. For any natural number $n$, 
we have the isomorphism $$\Inv: H^3(X, \mu_n)\simeq \nZ$$
and a
  perfect pairing \cite{Maz}
$$\<\cdot, \cdot \>: H^i(X, \cF)\times \Ext^{3-i}_X(\cF, \Gm)\ra H^3(X,\mu_n)\simeq \nZ$$
for any $n$-torsion sheaf $\cF$ in the \'etale topology\footnote{The pairing usually goes to $H^3(X, \Gm)\simeq\Q/\Z$. But the statement that it is perfect  means it induces an isomorphism $$\Ext^{3-i}_X(\cF, \Gm)\simeq \Hom(H^i(X, \cF), H^3(X, \Gm)).$$  But $H^i(X, \cF)$ is $n$-torsion, which means that the image of any homomorphism lies in the $n$-torsion subgroup $H^3(X, \Gm)[n]\simeq H^3(X, \mu_n)$.} .
\ms


The cup product
$$\cup: H^1(X, \zmod n)\times H^2(X, \mu_n)\ra H^3(X, \mu_n)\simeq \nZ$$
induces a map
$$r: H^2(X, \mu_n)\ra \Ext^2_X(\zmod n, \Gm)$$
such that
$$\Inv(a\cup b)=\pair{a}{r(b)}.$$
The Bockstein operator 
$$\d:  H^1(X, \mu_n)\ra H^2(X, \mu_n)$$
comes from the exact sequences of sheaves
$$0\ra \mu_n\ra \mu_{n^2} \ra \mu_n\ra 0.$$
Define the  coboundary map $d$ as the composition
$$d: H^1(X,\zmod n)\stackrel{\z_*}{\simeq}H^1(X,\mu_n)\stackrel{\d}{\ra} H^2(X, \mu_n)\stackrel{r}{\ra} \Ext^2_X(\zmod n, \Gm).$$
We view the two-term complex
$$H^1(X, \zmod n)\stackrel{d}{\ra} \Ext^2_X(\zmod n, \Gm)$$
as a mod $n$ arithmetic analogue of the complex
$$\Omega^1_M\ra \Omega^2_M$$
for three-manifolds.  The idea that cohomology equipped with the Bockstein operation can have the nature of differential forms occurs in the theory of the de Rham-Witt complex for a variety in characteristic $p$: there, the de Rham-Witt differentials are sheaves of crystalline cohomology \cite{IR}.  Also, recall that the curvature of a connection is the obstruction to deforming a bundle along a deformation of the space on which it lives. The Bockstein operator is a small piece of the obstruction to deforming it along a deformation of the coefficients.
\ms

There is also a Bockstein operator
$$\d': H^1(X, \zmod n)\ra H^2(X, \zmod n)$$
associated with the exact sequence
$$0\ra \zmod n \ra \Z/n^2\Z\ra \zmod n\ra 0$$
and a Bockstein in degree 2,
$$\delta_2: H^2(X, \mu_n)\ra  H^3(X, \mu_n).$$
By choosing an isomorphism $\Z/n^2\simeq \mu_{n^2}$ compatible with $\z$, we see an equality of maps
$$\z_*\circ \d'=\d\circ \z_*: H^1(X, \zmod n)\ra H^2(X, \mu_n).$$
The following fact is of course well-known, but it seems to be hard to find a reference for \'etale cohomology.
\begin{lem}The Bockstein operator $\delta_2$ satisfies
$$
\delta_2(\alpha \cup \beta) = \delta' \alpha  \cup \beta - \alpha \cup \delta \beta
$$
for all $\alpha \in H^1(X,\zmod n)$ and $\beta \in H^1(X, \mu_n)$. 
\end{lem}
\begin{proof}
Since $X$ is affine, the \'etale cohomology groups are isomorphic to the \v{C}ech cohomology groups (cf. \cite[Theorem 10.2]{Mil2}). Thus, we can check the above formula using \v{C}ech cocycles (cf. \cite[\ts 22]{Mil2}).

Choose a sufficiently fine \'etale covering $(U_i)_{i\in I}$ of $X$. Define $U_{ij} = U_i \times_X U_j$, $U_{ijk}= U_i \times_X U_j \times_X U_k$ and so on. Typical elements of the index set $I$ are denoted by $i,j,k,$ and $l$. Represent $\alpha$ and $\beta$ as \v{C}ech cocycles $(\alpha_{ij})$ and $(\beta_{ij})$. For any pair $(i,j)$ of distinct elements in $I$, choose a lifting $\tilde \alpha_{ij}$ of $\alpha_{ij}$ to $\zmod {n^2}$, and similarly a lifting $\tilde \beta_{ij}$ of $\beta_{ij}$ to $\mu_{n^2}$. The class of $\delta' \alpha$ can be represented by the $2$-cocycle whose section over $U_{ijk}$ is $(\delta' \alpha)_{ijk} := \tilde \alpha_{ij}|_{U_{ijk}} + \tilde \alpha_{jk}|_{U_{ijk}} - \tilde \alpha_{ik}|_{U_{ijk}}$ which takes values in $\zmod n \subset \zmod{n^2}$. We represent $\delta \beta$ in a similar way.

The cup product $\delta' \alpha \cup \beta$ is represented by a family of sections
$
\gamma'_{ijkl} = (\delta' \alpha)_{ijk}|_{U_{ijkl}} \otimes \beta_{kl}|_{U_{ijkl}}
$
and similarly $\alpha \cup \delta \beta$ is represented by
$
\gamma_{ijkl} = \alpha_{ij}|_{U_{ijkl}} \otimes (\delta \beta)_{jkl}|_{U_{ijkl}}.
$
On the other hand, we have
$$(\alpha \cup \beta)_{ijk} = \alpha_{ij}|_{U_{ijk}} \otimes \beta_{jk}|_{U_{ijk}}$$
which lifts to $\tilde \alpha_{ij}|_{U_{ijk}} \otimes \tilde\beta_{jk}|_{U_{ijk}}$ with values in $\zmod {n^2} \otimes \mu_{n^2} \simeq \mu_{n^2}$. A $\mu_n$-valued cocycle representing $\delta_2(\alpha \cup \beta)$ takes the form
$$
(\delta_2(\alpha\cup\beta))_{ijkl} := 
\left(\tilde \alpha_{jk} \cdot \tilde\beta_{kl}|_{U_{jkl}}\right)|_{U_{ijkl}} 
- \left(\tilde \alpha_{ik}\cdot \tilde\beta_{kl}|_{U_{ikl}}\right)|_{U_{ijkl}} 
+\left(\tilde \alpha_{ij}\cdot \tilde\beta_{jl}|_{U_{ijl}}\right)|_{U_{ijkl}} 
- \left(\tilde \alpha_{ij}\cdot  \tilde\beta_{jk}|_{U_{ijk}}\right)|_{U_{ijkl}} 
$$
where the isomorphism $\zmod {n^2} \otimes \mu_{n^2} \simeq \mu_{n^2}$ sends $a \otimes b \mapsto a \cdot b$ by viewing $\mu_{n^2}$ an additive group. Since $\alpha_{ij}$ and $\beta_{ij}$ are cocycles, $\tilde \alpha_{jk}|_{U_ijk} - \tilde \alpha_{ik}|_{U_ijk} = -\tilde \alpha_{ij}|_{U_ijk}+n\phi_{ijk}$ for some $\phi_{ijk}$ and similarly $\tilde \beta_{jl}|_{U_jkl} - \tilde \beta_{jk}|_{U_jkl} = \tilde \beta_{kl}|_{U_jkl} - n\psi_{jkl}$ for some $\psi_{jkl}$. Using these, the above simplifies to
$$
(\delta_2(\alpha\cup\beta))_{ijkl} = 
 \left((-\tilde \alpha_{ij} + n \phi_{ijk} )\cdot \tilde\beta_{kl}\right)|_{U_{ijkl}} 
+\left(\tilde \alpha_{ij}\cdot (\tilde \beta_{kl} - n \psi_{jkl}) \right)|_{U_{ijkl}} 
$$
$$= \left(n\phi_{ijk}\cdot \tilde\beta_{kl}\right)|_{U_{ijkl}} 
+\left(\tilde \alpha_{ij}\cdot (- n \psi_{jkl}) \right)|_{U_{ijkl}} 
$$
which is equal to $\gamma_{ijkl}' - \gamma_{ijkl}$ via the isomorphisms $\zmod n  \simeq n\Z / n^2\Z$ sending $a \mapsto na$, and $\mu_n \simeq \mu_{n^2}/\mu_{n}$ sending $\xi\mapsto \xi^{1/n}$. Hence we have shown the desired property of $\delta_2$. 
\end{proof}

Define
the pairings
$$(\cdot, \cdot ): H^1(X, \zmod n)\times H^1(X,\zmod n)\ra \nZ;$$
$$(\a,\b):=\pair{\a}{d\b}\in \nZ.$$

\begin{lem}\label{lem11}
The pairing is symmetric:
$$(\a, \b)=(\b,\a)$$
for all $\a, \b \in  H^1(X,\zmod n)$.
\end{lem}
\begin{proof}
This follows from examining the second Bockstein operator above.
$$\d_2: H^2(X, \mu_n)\ra H^3(X, \mu_n).$$
For the pro-sheaf $\Z_n(1):=\invlim_{i}\mu_{n^i}$, we have an exact sequence
$$0\ra \Z_n(1)\stackrel{n}{\ra}\Z_n(1)\ra \mu_n\ra 0.$$
Because $H^3(X, \Z_n(1))\simeq \Z_n$ is torsion-free,  the boundary map
$H^2(X, \mu_n)\ra H^3(X, \Z_n (1))$ is zero, and the map
$H^2(X, \Z_n(1))\ra H^2(X, \mu_n)$ is surjective. Hence, $H^2(X, \mu_{n^2})\ra H^2(X, \mu_n)$ is surjective, so that the map $\d_2$ is zero.

As a consequence, we have
$$0=\d_2(\a\cup \b)=\d'\a\cup \b -\a \cup \d \b.$$
Therefore,
$$(a,b)=\pair{a}{db}=\Inv(a\cup \d \z_*b)=\Inv(\d' a\cup \z_* b)=\Inv(\z_*(\d' a\cup  b))$$
$$=\Inv (\z_*(b\cup \d' a))=\Inv (b\cup \z_*\d' a))=
\Inv (b\cup \d\z_* a))=\pair{b}{da}=(b,a).$$
\end{proof}

Define $K=\Ker(d)$.
\begin{cor} \label{cor: annihilator}
If $a\in K$, then $(a,b)=0$ for all $b$.
\end{cor}
\begin{proof} If $a\in K$, then
$(a,b)=(b,a)=\pair{b}{da}=0$. 
\end{proof}

According to  duality, we have
$\Ext^2_X(\zmod n, \Gm)\simeq H^1(X,\zmod n)^{\vee}\simeq Cl(X)/n$, where $Cl(X)$ is the ideal class group of $X$.
We will say an ideal $I\subset \cO_F$ is \textit{$n$-homologically trivial} if
its class in $\Ext^2_X(\zmod n, \Gm)$ is in the image of $d$. Even though there is some danger of confusion, when $n$ is fixed for the discussion, we will also allow ourselves merely to say that $I$ is  `homologically trivial'. If $I$ and $J$ are homologically trivial ideals, we define \textit{the mod $n$ height pairing} between $I$ and $J$ by
$$ht_n(I,J)=(d^{-1}[I]_n,[J]_n),$$
where $[I]_n$ denotes the class of $I$ in $Cl(X)/n$. 
Writing $[J]_n=d(b)$ for some $b\in H^1(X,\zmod n)$,  for any $a$ such that $da=0$, we have $\pair{a}{db}=(a, b )=0$ by Corollary \ref{cor: annihilator}. This implies that the mod $n$ height pairing is well-defined. 
Using the pairing on $H^1(X, \zmod n)$, note that we can also write the height pairing as
$$(d^{-1}[I]_n, d^{-1}[J]_n),$$
rendering the symmetry evident.
For two prime ideals $\cP$ and $\cQ$ (which are homologically trivial), 
we will also call their height pairing \textit{their linking number}, and denote it
$$\lk_n(\cP, \cQ):=ht_n(\cP, \cQ)=\<d^{-1}[\cP]_n, [\cQ]_n\>.$$
\ms

In the  papers \cite{Kim, CKKPY}, we fixed  a class $c \in H^3(\zmod n, ~\zmod n)$ and defined the \textit{arithmetic Chern-Simons action} for homomorphisms
$$\rho : \pi=\pi_1(X, b) \ra \zmod n$$ as
\[
CS_c(\rho) := \Inv (\z_*(j^3(\rho^*(c)))) \in \modz n,
\]
where 
$j^i : H^i(\pi, ~\zmod n) \ra H^i (X, ~\zmod n)$ is the natural map from group cohomology to \'etale cohomology (cf. \cite[Theorem 5.3 of Chap. I]{Mil}). 
We can also define the \textit{arithmetic Chern-Simons partition function} as
\[
Z_c(X) := \sum_{\rho \in \Hom (\pi, ~\zmod n)}~ \exp~(2\pi i \cdot CS_c(\rho)).
\]

The class $c:=Id \cup \tilde{\d}(Id)$ is a generator of $H^3(\zmod n,~\zmod n)$, where
$Id$ is the identity from $\zmod n$ to $\zmod n$ regarded as an element of $H^1(\zmod n, \zmod n)=\Hom(\zmod n, \zmod n)$ and $\tilde{\d} : H^1(\zmod n, ~\zmod n) \ra H^2(\zmod n, ~\zmod n)$ is a Bockstein operator induced from the  exact sequence
\[
\sesdot {\zmod n}{\zmod {n^2}}{\zmod n}
\]
There is a natural bijection between $\Hom (\pi, \zmod n)$ and $H^1(X,~\zmod n)$ (defined by $j^1$) and we will simply identify the two. One then checks immediately that for the cocycle $c=Id \cup \tilde{\d} (Id)$ we have
$$CS_c(\rho)=(\rho, \rho).$$
Thus for the partition function, we have
\begin{equation*}\label{eqn: partition function}
Z_c(X)= \sum_{\rho \in \Hom(\pi, ~\zmod n)} ~\exp~(2\pi i \cdot CS_c(\rho))=\sum_{\rho \in H^1(X,~\zmod n)}~\exp~(2\pi i \cdot (\rho,\rho)).
\end{equation*}
\ms

\begin{proof}[Proof of Theorem \ref{thm:maintheorem}]
By Corollary \ref{cor: annihilator} and the definition $(\rho,\rho )=\pair{\rho}{d\rho}$, both $(\rho,\rho )$ and $\pair{\rho}{[{\xi_j}]_p}$ depend only on the class of $\rho$ in $H^1(X,\zmod p)/K$, which we denote by $\bar{\rho}$. So we can write the sum as
$$p^b\sum_{\bar{\rho}\in H^1(X, ~\zmod p)/K} \exp[2\pi i ((\bar{\rho}, \bar{\rho})+\sum_j\pair{\bar{\rho}}{[\xi_j]_p})].$$
After a choice of basis for $H^1(X,\zmod p)/K$ and $\Im(d)$, this becomes a Gaussian integral over a finite field. Now the formula follows from \cite[Proposition 3.2 of Chap. 9]{Ner}.
\end{proof}

\section{Boundaries}\label{section3}
In this section, we fix a natural number $n$ and a finite set $S$ of places of $F$ containing all the places that divide $n$ and the Archimedean places.  As before, we assume $\mu_{n^2}\subset F$. Put $U=\Spec(\cO_{F,S})$, the spectrum of the ring of $S$-integers in $F$. Let $\pi_U:=\pi_1(U)$ and $\pi_v:=\pi_1(\Spec(F_v))$ for each place $v$ of $F$. Denote by $C^*(U, \cG)$ the complex of continuous cochains of $\pi_U$ with coefficients in a locally constant torsion $\Z_n=\invlim_i \Z/n^i\Z$-sheaf $\cG$ on $U$ and by $C^*(F_v,\cF)$, the complex of continuous cochains of $\pi_v$ with coefficients in a sheaf  $\cF$ on $\Spec(F_v)$.
As in \cite[\textsection 2]{CKKPY}, we will use the `inclusion of the boundary' map
$$i_S=\prod_{v\in S} i_v : \bU=\coprod_{v\in S} \Spec(F_v)\ra U.$$  Let $\cG$ be a sheaf on $U$, $\cF$ a sheaf on $\bU$, and
$f:\cF\ra i_S^*\cG$  a map of sheaves. In view of the applications in mind, we will refer to such a map as a {\em boundary pair}.
Denote by $C^*(U, \cG\times_S \cF)$, the two product of complexes defined by the following diagram:
\[
\xyh{0.4}
\xyv{0.4}
\xymatrix{
C^*(U, \cG\times_S \cF) \ar[rr] \ar[dd] &&\prod_{v\in S} C^*(F_v, \cF) \ar[dd]^-{f_*} \\
&\fbox{2}&
\\
C^*(U, \cG) \ar[rr]_-{loc_S} &&  \prod_{v\in S}C^*(F_v, i_v^*\cG),
}
\]
where $loc_S$ refers to the localisation map on cochains. Thus, $$C^i(U, \cG\times_S \cF)= C^i(U, \cG)\times \prod_{v\in S} C^i(F_v, \cF) \times \prod_{v\in S}C^{i-1}(F_v, i_v^*\cG),$$
and its elements will be denoted by $(c, b_S, a_S)$, where $c\in C^i(U, \cG)$, $b_S=(b_v)_{v\in S}\in \prod_{v\in S} C^i(F_v, \cF)$,
and $a_S=(a_v)_{v\in S}\in \prod_{v\in S}C^{i-1}(F_v, i_v^*\cG)$. The differential is defined by
$$d(c,b_S,a_S)=(dc,db_S, da_S+(-1)^i(f_*(b_S)-loc_S(c)).$$
Hence, a cocycle in $Z^i(U, \cG\times_S \cF)$ consists of $(c,b_S, a_S)$ such that $dc=0, db_S=0$, and
$$da_S=(-1)^i(loc_S(c)-f_*(b_S)).$$
Define
$$H^i(U, \cG\times_S \cF):=H^i( C^*(U, \cG\times_S \cF)).$$

Here are some general properties that follow immediately from the definitions.
\ms

(1) When $\cF=0$, then $H^i(U, \cG\times_S 0)=H^i_c(U, \cG)$, the compact support cohomology of $\cG$.
\ms

(2) Given maps $\cF\ra \cF'$, $\cG\ra \cG'$ and a commutative diagram
\[
\xymatrix{
\cF \ar[r] \ar[d] & \cF' \ar[d]\\
i_S^*\cG \ar[r] &i_S^*\cG'
}
\]
we have an induced map of complexes $$C^*(U, \cG\times_S \cF)\ra  C^*(U, \cG'\times_S \cF'),$$ and hence, a map of cohomologies
$$H^i(U, \cG\times_S \cF)\ra H^i(U, \cG'\times_S \cF').$$
More precisely, the formation of the complex and the cohomology is functorial in the diagrams in an obvious sense.
\ms

(3) Suppose you have two exact sequences
$$0\ra  \cF''\ra \cF\ra  \cF' \ra 0$$
$$0\ra  \cG''\ra \cG \ra \cG'\ra 0$$
and a commutative diagram
\[
\xymatrix{
0 \ar[r]  & \cF'' \ar[r] \ar[d] & \cF \ar[r] \ar[d] & \cF' \ar[r] \ar[d] & 0 \\
0 \ar[r] & i_S^* \cG'' \ar[r] & i_S^* \cG \ar[r] & i_S^* \cG' \ar[r] & 0.
}
\]
Then you get an exact sequence of complexes
$$0\ra C^*(U, \cG''\times_S \cF'')\ra C^*(U, \cG\times_S \cF)\ra C^*(U, \cG'\times_S \cF')\ra 0,$$
and hence, a long exact sequence at the level of cohomology.
\ms


(4) Cup product is given by
$$C^i(U, \cG\times_S \cF)\times C^j(U, \cG'\times_S \cF')\ra C^{i+j}(U, (\cG\otimes \cG' )\times_S (\cF\otimes \cF'))$$
$$(c,b_S, a_S)\cup (c',b_S', a_S')=(c\cup c', b_S\cup b'_S, (-1)^ja_S\cup f_*(b_S')+loc_S(c)\cup a_S').$$
Another possibility for the cup product, temporarily denoted by $\cup'$, is
$$(c,b_S, a_S)\cup' (c',b_S', a_S')=(c\cup c', b_S\cup b'_S, (-1)^ja_S\cup loc_S(c')+f_*(b_S)\cup a_S').$$
The difference is
$$\Delta=(0,0, (-1)^ja_S\cup (f_*(b_S')-loc_S(c'))+(loc_S(c)-f_*(b_S))\cup a'_S).$$
It will be useful to note that
\begin{lem} \label{lem21} When the two cochains are cocycles, the difference above is exact.
\end{lem}
\begin{proof}
The cocycle condition says
$$da_S=(-1)^i(loc_S(c)-f_*(b_S)); \ \ \ \ \ \ da'_S=(-1)^j(loc_S(c')-f_*(b'_S)).$$
Hence,
$$d(a_S\cup a_S')=(-1)^i(loc_S(c)-f_*(b_S))\cup a'_S+(-1)^{i+j-1} a_S\cup (loc_S(c')-f_*(b'_S))$$
$$=(-1)^i(loc_S(c)-f_*(b_S))\cup a'_S+(-1)^{i+j} a_S\cup (f_*(b'_S)-loc_S(c'))$$
$$=(-1)^i\left((loc_S(c)-f_*(b_S))\cup a'_S+(-1)^{j} a_S\cup (f_*(b'_S)-loc_S(c'))\right).$$
Hence,
$$d(0,0,(-1)^ia_S\cup a'_S)=\Delta.$$
\end{proof}
The differential is compatible with the cup product:

\begin{lem} If $(c,b_S, a_S)\in C^i$ and $(c',b_S', a_S')\in C^j$, then
$$d[(c,b_S, a_S)\cup (c',b_S', a_S')]=[d(c,b_S, a_S)]\cup (c',b_S', a_S')+(-1)^i(c,b_S, a_S)\cup [d(c',b_S', a_S')].$$

\end{lem}
\begin{proof}
We have
$$d[(c,b_S, a_S)\cup (c',b_S', a_S')]=d(c\cup c', b_S\cup b_S', (-1)^ja_S\cup f_*(b_S')+loc_S(c)\cup a_S')$$
$$=(dc\cup c'+(-1)^ic\cup dc', db_S\cup b'_S+(-1)^ib_S\cup db'_S, (-1)^jda_S\cup f_*(b'_S)+(-1)^{i+j-1}a_S\cup f_*(db'_S)$$
$$+loc_S(dc)\cup a'_S+(-1)^iloc_S(c)\cup da'_S+(-1)^{i+j}(f_*(b_S\cup b'_S)-loc_S(c\cup c')),$$
where the last component is the only thing we need to focus on. On the other hand, we have
$$d(c,b_S, a_S)\cup (c',b_S', a_S')=(dc, db_S, da_S+(-1)^i(f_*(b_S)-loc_S(c))\cup (c',b_S', a_S')$$
$$=(dc\cup c', db_S\cup b_S', (-1)^j (da_S+(-1)^i(f_*(b_S)-loc_S(c)))\cup f_*(b'_S)+loc_S(dc)\cup a'_S).$$
Also,
$$(c,b_S, a_S)\cup d(c',b_S', a_S')=(c,b_S,a_S)\cup (dc', db'_S, da'_S+(-1)^j(f_*(b'_S)-loc_S(c')))$$
$$=(c\cup dc', b_S\cup db'_S, (-1)^{j+1}a_S\cup f_*(db'_S)+loc_S(c)\cup [da'_S+(-1)^j(f_*(b'_S)-loc_S(c'))]).$$
So the third component of
$$d(c,b_S, a_S)\cup (c',b_S', a_S')+(-1)^i(c,b_S, a_S)\cup d(c',b_S', a_S')$$
is
$$(-1)^j (da_S+(-1)^i(f_*(b_S)-loc_S(c)))\cup f_*(b'_S)+loc_S(dc)\cup a'_S$$
$$+(-1)^i((-1)^{j+1}a_S\cup f_*(db'_S)+loc_S(c)\cup [da'_S+(-1)^j(f_*(b'_S)-loc_S(c'))))$$
$$=(-1)^jda_S\cup f_*(b'_S)+(-1)^{i+j}(f_*(b_S\cup b'_S)+(-1)^{i+j-1}loc_S(c)\cup f_*(b'_S)+loc_S(dc)\cup a'_S$$
$$+(-1)^{i+j-1}a_S\cup f_*(db'_S)+(-1)^iloc_S(c)\cup da_S'+(-1)^{i+j}loc_S(c)\cup f_*(b'_S)+(-1)^{i+j-1}loc_S(c\cup c').$$
$$= (-1)^jda_S\cup f_*(b'_S)+(-1)^{i+j}(f_*(b_S\cup b'_S)+loc_S(dc)\cup a'_S$$
$$+(-1)^{i+j-1}a_S\cup f_*(db'_S)+(-1)^iloc_S(c)\cup da_S'+(-1)^{i+j-1}loc_S(c\cup c').$$
This is easily seen to be the third component of $d[(c,b_S, a_S)\cup (c',b_S', a_S')]$ above.
\end{proof}

\begin{cor}
The cup product of cocycles is a cocycle.
\end{cor}

\begin{cor}
The cup product of cocycles  induces a graded product map on cohomologies.
\end{cor}
\begin{proof}
Of course this is because if $\b$ is a cocycle, then $d(\a\cup \b)=\pm (d\a)\cup \b$ is a coboundary.
\end{proof}

The main case of interest is when  $\cF=\mu_{n^2}$, $\cG=\zmod n$ and $f:\mu_{n^2} \ra \zmod n$ is the natural reduction followed by the trivialisation $\zeta^{-1}: \mu_n \simeq \zmod n$.
From the exact sequence of pairs
\[
\xymatrix{
0 \ar[r]  & 0 \ar[r] \ar[d] & \mu_{n^2} \ar[r]^-{Id} \ar[d] & \mu_{n^2} \ar[r] \ar[d]^-f & 0 \\
0 \ar[r] & \zmod n \ar[r]^-{\z_*^{-1}} & \mu_{n^2} \ar[r]^-f & \zmod n \ar[r] & 0
}
\]
and (1), (3) above,
we get  natural boundary maps
 $$d: H^1(U, \znn)\ra H^2_c(U, \zmod n)$$
 and
 $$d_2:H^2(U, \znn)\ra H^3_c(U, \zmod n).$$
 \begin{prop}
 The map $d_2$ is zero.
 \end{prop}
 \begin{proof}
 We will show that the previous map
 $$H^2(U, \mu_{n^2}\times_S \mu_{n^2})\ra H^2(U, \znn)$$
 in the long exact sequence of cohomology is surjective. First we note that the map
 $$H^2(U, \mu_{n^2})\ra H^2(U, \zmod n)$$ is surjective. To see this, use the map of exact sequences
\[
\xymatrix{
0 \ar[r]  & H^2(U, \Gm)[n^2] \ar[r] \ar[d]^-n & \prod_{v\in S} \n2Z \ar[r]^-{\Sigma} \ar[d]^-n & \prod_{v\in S} \n2Z \ar[r] \ar[d]^-n & 0 \\
0 \ar[r] & H^2(U, \Gm)[n^2] \ar[r] & \prod_{v\in S} \nZ \ar[r]^-\Sigma & \prod_{v\in S} \nZ \ar[r] & 0
}
\] 
 from class field theory. The sum map is surjective from the kernel of the middle vertical map to the kernel of the right vertical map.
 The middle vertical map is trivially surjective. Hence, the left vertical map is surjective by the snake lemma.
 On the other hand, we also have the map of exact sequences
\[
\xymatrix{
0 \ar[r]  & H^1(U, \Gm)/{n^2H^1(U,\Gm)} \ar[r] \ar[d] & H^2(U, \mu_{n^2}) \ar[r] \ar[d]^-n & H^2(U, \Gm)[n^2] \ar[r] \ar[d]^-n & 0 \\
0 \ar[r] & H^1(U, \Gm)/nH^1(U,\Gm) \ar[r] & H^2(U, \mu_{n}) \ar[r] & H^2(U, \Gm)[n] \ar[r] & 0,
}
\]
 where the left vertical map is the natural projection. Since the vertical maps on the left and right are surjective, so is the one in the middle.
 
 Now let $(c, b_S, a_S)\in Z^2(U, \znn)$. 
 Choose $c'\in Z^2(U, \mu_{n^2})$ lifting $c$ and 
 $a'_S\in \prod_{v\in S} C^1(F_v, \mu_{n^2})$ lifting $a_S$ under the map $f:\mu_{n^2} \to \zmod n$. Then
 $$da'_S=b_S-loc_S(c')+(b_S')^n$$
 for some $b'_S\in C^2(U, \mu_{n^2})$. However, $(b_S')^n$ is a cocycle, since this is true of all other terms in the equality. Hence, 
 $$(c', b_S+(b'_S)^n, a'_S)\in Z^2(U, \mu_{n^2}\times_S\mu_{n^2})$$
 is a lift of $(c, b_S, a_S)$.
 \end{proof}

\begin{lem}
For $\a \in H^1(U, \znn)$ and $\b \in H^2_c(U, \zmod n)$, we have
$$\a\cup \b=\b\cup \a $$
in $H^3_c(U, \zmod n).$
\end{lem}
\begin{proof}
Choose cocycle representatives
$(c,b_S,a_S)$ and $(c', 0, a_S')$ for $\a$ and $\b$.
Then
$$\a\cup \b=[ (c,b_S,a_S)\cup (c', 0, a_S')]=[(c\cup c', 0, loc_S(c)\cup a'_S)]$$
and 
$$\b\cup \a=[ (c'\cup c, 0, -a'_S\cup b_S+loc_S(c')\cup a_S)]=[ (c'\cup c, 0, -a'_S\cup loc_S(c))].$$
Here the last equality follows from Lemma \ref{lem21}.
Note that we have the map
$$\eta: H^1(U,\znn)\ra H^1(U, \zmod n)$$
that sends $(c, b_S, a_S)$ to $c$.
So, using $\eta$, the desired commutativity $\a\cup \b=\b\cup \a$ reduces to the commutativity of the following two
products:
$$\begin{array}{ccccccc}
H^1(U, \zmod n)\times H^2_c(U, \zmod n)&\stackrel{\cup}\longrightarrow& H^3_c(U, \zmod n)& \\
([(c,0,0)], [(c',0,a'_S)])&\longmapsto&[(c\cup c', 0, loc_S(c)\cup a'_S)]&
\end{array}
$$
 
$$\begin{array}{ccccccc}
H^2_c(U, \zmod n)\times H^1(U, \zmod n)&\stackrel{\cup}\longrightarrow& H^3_c(U, \zmod n)& \\
([(c',0,a'_S)], (c,0,0)])&\longmapsto&[ (c'\cup c, 0, -a'_S\cup loc_S(c))].&
\end{array}
$$

These products are the same as the ones defined in \cite[\textsection 5.3.3]{Nek}. Moreover, Nekovar defined the involution
$$
\cT: C^*(U, \zmod  n) \ra C^*(U, \zmod n), \quad \cT: C^*(U, \zmod n\times_S 0) \ra C^*(U, \zmod n\times_S 0),
$$
which are homotopic to the identity, and showed that the following diagram is commutative:
\[
\xymatrix{
C^*(U, \zmod n)\times C^*(U, \zmod n \times_S 0 ) \ar[r]^-\cup \ar[d]^-{s_{12}\circ (\cT\otimes \cT)} & C^*(U, \zmod n  \times_S 0) \ar[d]^-{\cT\circ (s_{12})_*}\\
C^*(U, \zmod n  \times_S 0)\times C^*(U, \zmod n) \ar[r]^-\cup & C^*(U, \zmod n  \times_S 0),
}
\]
where $s_{12}$ is the permutation between $\zmod n$ and $\zmod n  \times_S 0$ defined similarly as in \cite[3.4.5.4]{Nek}.
This finishes the proof.
%
\end{proof}

The proof of the previous lemma makes use of the natural map 
$$\eta: H^1(U, \zmod n\times_S \mu_{n^2}) \to H^1(U, \zmod n)$$
that sends $(c, b_S, a_S)$ to $c$. In fact, we have proved:
\begin{lem}\label{lem27}
The cup product
$$ H^1(U, \zmod n\times_S \mu_{n^2})\times H^2_c(U, \zmod n) \to H^3_c(U, \zmod n)$$
factors through the product
$$H^1(U, \zmod n)\times H^2_c(U, \zmod n) \to H^3_c(U, \zmod n)$$
via the map $\eta$. This is also true with the factors switched.
\end{lem}

We now use the tools developed above to define a pairing
$$H^1(U, \znn)\times H^1(U,\znn)\ra \nZ$$
given by 
$$(a,b )=\Inv\circ \zeta_* (a\cup d b).$$

The pairing goes through
$$H^1(U, \znn)\times H^2(U, \zmod n\times_S 0)\simeq H^1(U, \znn)\times H^2_c(U, \zmod n),$$
and hence, through $H^3(U, \zmod n\times_S 0)=H^3_c(U, \zmod n)\stackrel{\zeta_*}\ra
H^3_c(U, \mu_n)\simeq \nZ.$
\begin{lem}
The pairing is symmetric. 
\end{lem}
\begin{proof}

We have $0=d_2(a\cup b)=da\cup b-a\cup db.$ So
$$(a,b)=\Inv\circ \zeta_* (a\cup db)=\Inv \circ \zeta_* (da\cup b)=\Inv \circ \z_*(b\cup da)=(b,a).$$
\end{proof}

\begin{lem}\label{lem28}
If $a\in \Ker(d)$, then
$$(a,b)=0$$
for all $b$.
\end{lem}
\begin{proof}
$$(a, b)=(b,a)=\Inv \circ \z_*(b\cup da)=0.$$
\end{proof}


Denote by $H^1(U,[\zmod n]')\subset H^1(U,\zmod n)$ the classes that locally (at all $v\in S$)  lift to $\mu_{n^2}$. Equivalently, $H^1(U,[\zmod n]')$ is the image of $\eta$. 
Because the pairing $(\cdot,\cdot)$ is symmetric and given by the form $(a,b) = \Inv\circ\zeta_*(\eta(a) \cup db)$ by Lemma~~\ref{lem27}, it follows that it factors to a pairing
$$H^1(U,[\zmod n]')\times H^1(U,[\zmod n]')\ra \nZ.$$


\ms


Let $A_S=(\pi_U)^{ab}$ be the maximal abelian quotient of $\pi_U$. By the Poitou-Tate duality, we have $H^2_c(U,\zmod n)\simeq A_S/n$. Given an ideal $I$ coprime to $S$, we can consider its class
$[I]_{S,n}\in H^2_c(U,\zmod n)$ via class field theory and the previous isomorphism. We will say $I$ is \textit{$(S,n)$-homologically trivial} if $[I]_{S,n}$ is in the image of $d$. 
We can now define \textit{the height pairing} of two $(S,n)$-homologically trivial ideals that are coprime to $S$ via
$$ht_{S,n}(I,J):=(d^{-1}[I]_{S,n}, d^{-1}[J]_{S,n})=\Inv \circ \zeta_* (d^{-1}[I]_{S,n} \cup [J]_{S,n})=:\pair{d^{-1}[I]_{S,n}}{[J]_{S,n}},$$
which is well-defined by the discussion above.

Let $I$ be an ideal such that $I^n$ is principal in $\cO_{F,S}$. Write
\vio{$I^n=(f^{-1})$}. Then the Kummer cocycles $k_n(f)$ will be in $Z^1(U, \zmod n)$. For any $a\in F$, denote by $a_S$ its image in
$\prod_{v\in S} F_v$.  Thus, we get an element
$$[f]_{S,n}:=[(k_n(f), k_{n^2}(f_S), 0)] \in Z^1(U, \zmod n \times_S \mu_{n^2})$$
which is well-defined in cohomology independently of the choice of roots used to define the Kummer cocycles.
\begin{prop}\label{prop210}
We have $d [f]_{S,n}=[I]_{S,n}$ in $H^2_c(U, \zmod n).$ In particular, for any ideal $I$ such that $I^n$ is principal in $\cO_{F,S}$, $[I]_{S,n}$ is $(S,n)$-homologically trivial.

\end{prop}

\begin{proof}

Let $T=S\cup S'$ be large enough that for $V=U\setminus S'$, $H^1(V, \Gm)[n]=0$, and such that the support of $I$ is still in $V$. Then
$I$ defines a class $[I]_{T,n}$  in $H^2_c(V, \zmod n)$. Similarly, $f$ defines a class $[f]_{T,n}=[(k_n(f), k_{n^2}(f_T),0)]$ in
$H^1(V, \zmod n \times_S \mu_{n^2})$. It is clear that the elements $[I]_{T,n}$ and $d[f]_{T,n}$ map to $[I]_{S,n}$ and $d[f]_{S,n}$
under the pushforward map $H^2_c(V, \zmod n)\ra H^2_c(U, \zmod n)$. 
Hence, it suffices to prove equality of the elements on $V$. We will prove that the two elements pair the same way with elements of $H^1(V,\mu_n)$.
On $V$, by the exact sequence
$$0\ra \Gm(V)/\Gm(V)^n\ra H^1(V, \mu_n)\ra H^1(V, \Gm)[n]\ra 0,$$
every element of $H^1(V,\mu_n)$ comes from $g\in \Gm(V)$ via the Kummer map.
For this, we can compute the pairing between $\b=(k_n(g), k_{n^2}(g_T), 0)$, which lifts note that $[k_n(g)]$ along $\eta$, and the cocycle representative $\a$ of $d [f]_{T,n}$
$$\a=(d\tkf, 0, (\tkf)_T-k_{n^2}(f_T)),$$
where $\tkf$ is a lift of $k_n(f)$ to $\mu_{n^2}$.
We find
$$\b\cup \a=( k_n(g)\cup d\tkf, 0, -(k_n(g))_T\cup [ (\tkf)_T-k_{n^2}(f_T) ]).$$
We note that the cup product $k_n(g)\cup \tkf$ takes values in $\zmod n \simeq \mu_n \subset \mu_{n^2}$. 
So we have the cochain
$(k_n(g)\cup \tkf,0,0)$ whose differential is $(k_n(g)\cup d\tkf, 0, -(k_n(g))_T\cup (\tkf)_T)$. Hence, it suffices to compute the invariant of
$$(0,0,   (k_n(g))_T\cup k_{n^2}(f_T) )$$
which is homologous to $\beta \cup \alpha$.

Let $T'$ be the support of $I$. Then $T\cup T'$ is the full set of places where the global cocycle $(k_n(g))_T\cup k_{n^2}(f_T)$ with coefficients in
$\mu_n\subset \mu_{n^2}$ is possibly ramified. By global reciprocity, we have
$$\sum_{v\in T}\Inv_v((k_n(g))_T\cup k_{n^2}(f_T) )=-\sum_{v\in T'} \Inv_v ((k_n(g))_{T'}\cup k_{n^2}(f_{T'})).$$
Let $\mathrm{ord}_v(I)=e_v$ and let $\varpi_v$ be a uniformizer at $v$. Then $f_v=u_v\varpi_v^{ne_v}$ for a unit $u_v \in F_v$, so that
$k_{n^2}(f_v)=k_n(u_v\varpi_v^{e_v})$. Also, $F(\sqrt[n]g)$ is unramified at $v\in T'$. Hence, for $v\in T'$, we get
$$ \Inv_v ((k_n(g))_{T'}\cup k_{n^2}(f_{T'}))=(g_v, u_v\varpi_v^{e_v})_n=(g_v, \varpi_v^{e_v})_n,$$
where the bracket $(\cdot, \cdot)_n$ now refers to the $n$-th Hilbert symbol.

Therefore, we conclude that
$$\sum_{v\in T} \Inv_v (k_n(g)_{T'}\cup k_{n^2}(f_{T'})) \vio{=- \sum_{v\in T'} k_n(g)_v(\mathrm{rec}_v(\varpi_v^{e_v}))=k_n(g)(\mathrm{rec}(I))=}\langle [\beta], [I]_{T,n}\rangle,$$
where $\mathrm{rec}_v$ is the local Artin map and $\mathrm{rec}$ is the global Artin map (cf. \cite[p. 174--176]{CF}), 
finishing the proof. \end{proof}

\begin{cor} \label{cor:3.11}
Let $I, J$ be  ideals  in $\cO_F$ supported outside $S$ that are $n$-torsion in the Picard group of $U$. Choose any $f \in F^*$ such that \vio{$I^n = (f^{-1})$} as ideals of $\cO_{F,S}$. Let $T$ be the support of $J$, $\varpi_v$ be a uniformiser at $v$, and $e_v=\mathrm{ord}_v(J)$. Then
$$ht_{S,n}(I,J)=\sum_{v\in T} (f_v, \varpi_v^{e_v})_n,$$
where the bracket denotes the $n$-th Hilbert symbol. 
\end{cor}
\begin{proof}
By Proposition~~\ref{prop210}, we have $[f]_{S,n} \in H^1(U, \zmod n \times_S \mu_{n^2})$ such that $\eta([f]_{S,n}) =k_n(f)$ and $d [f]_{S,n} = [I]_{S,n} \in \pi_U^{ab} \simeq H^2_c(U,\zmod n)$. The pairing $ht_{S,n}(I,J)$ is given by the Poitou-Tate pairing $\langle k_n(f), [J]_{S,n} \rangle$, which is equal to $k_n(f)([J]_{S,n}) = \sum_{v \in T} (f_v,\varpi_v^{e_v})_n$ by the local-global compatibility of Artin maps. 
\end{proof}





\section{Arithmetic linking, class invariants and the Artin map}\label{section4}

We continue with the assumption of a fixed  trivialization $\zeta:\Z/n\Z\simeq\mu_n$ over the totally imaginary number field $F$. 

Let us recall the construction of the class invariant homomorphism
\[
\Psi: H^1(X,\Z/n\Z)\to \Cl(X)
\]
of Waterhouse \cite{Wat} and M. Taylor \cite{Tay}. Suppose $x\in H^1(X, \Z/n\Z)$ is the class of the $\Z/n\Z$-torsor given 
as the spectrum of an \'etale $\O_F$-algebra $\O$ with $\Z/n\Z$-action. To avoid confusion we will write 
$\sigma_a(v)$ for the effect of the action of $a\in \Z/n\Z={\rm Gal}(\O/\O_F)$ on $v\in \O$. We consider the $\O_F$-module ${\calL}$ consisting of all elements $v\in \O$ such that
\[
\sigma_a( v) = \zeta(a)\cdot   v 
\]
for all $a\in \Z/n\Z$. Using \'etale descent along the extension $\O/\O_F$ we can easily see that 
${\calL}$ is $\O_F$-locally free of rank $1$. Then we set $\Psi(x)=\Psi(\O/\O_F)$ to be the class of $\calL$ in $\Pic(X)=\Cl(X)$.
This homomorphism $\Psi$ can also be viewed as follows: 
The $\Z/n\Z\simeq\mu_n$-torsor over $X$
that corresponds to $x$ induces by $\mu_n\to \Gm$ a $\Gm$-torsor, i.e.
a line bundle whose class is $\Psi(x)$. 

This construction plays a central role in the theory
of Galois module structure; indeed, $\Psi(x)$ is an important invariant of the structure 
of $\O$ as an $\O_F[\Z/n\Z]=\O_F[x]/(x^n-1)$-module. The general form of the class invariant homomorphism 
for the constant group scheme $\Z/n\Z$ with Cartier dual $\bd{\mu}_n$ is
\[
H^1(X, \Z/n\Z)\to {\rm Pic}((\bd{\mu}_n)_{X})={\rm Pic}(\O_F[x]/(x^n-1)).
\]
(See, for example, \cite{Wat}.)
The map $\Psi$ above is obtained by composing the above with the restriction along the
section
$X\to (\bd{\mu}_n)_X$ given by $x\mapsto \zeta(1)$.

Combining this with class field theory allows us to define the {\rm class invariant} pairing 
\[
(\cdot , \cdot )_c: (\Cl(X)/n)^\vee\times (\Cl(X)/n)^\vee\to \Z/n\Z
\]
as follows: Take $f$, $f' \in  (\Cl(X)/n)^\vee={\rm Hom}_{\Z}(\Cl(X), \Z/n\Z)$. By class field theory,
$f$ and $f'$ correspond to unramified $\Z/n\Z$-extensions $K_f$ and $K_{f'}$ of $F$. Let $\O_f$ and $\O_{f'}$ be the normalisations of $\O_F$ in $K_f$ and $K_{f'}$ respectively; these are \'etale $\O_F$-algebras with $\Z/n\Z$-action. By definition, the class invariant pairing is
\[
(f, f')_c:=f'(\Psi(\O_f/\O_F)).
\]

\begin{thm}\label{cithm}
 Under the class field theory isomorphism
\[
 {\mathrm {Ar}}:  H^1(X, \Z/n\Z)\xrightarrow{\sim} (\Cl(X)/n)^\vee,
 \]
the class invariant pairing
\[
(\cdot ,\cdot )_c: (\Cl(X)/n)^\vee\times (\Cl(X)/n)^\vee\to \Z/n\Z
\]
is identified with the   pairing 
\[
(\cdot,\cdot): H^1(X, \Z/n\Z)\times H^1(X, \Z/n\Z)\to \frac{1}{n}\Z/\Z=\Z/n\Z, \quad (\alpha, \beta)={\rm Inv} \circ\zeta_*(\alpha\cup \delta'\beta),
\]
 defined as in Section \ref{section2}. 
\end{thm}

\begin{rem}$~$
\begin{enumerate}[a)]
\item
It follows that  the arithmetic Chern-Simons invariant 
\[
CS_c: H^1(X, \Z/n\Z)\to \Z/n\Z, \quad CS_c(x)=(x, x),
\]
for $c=Id\cup \tilde{\delta} (Id)$ can be identified under  $ (\Cl(X)/n)^\vee\simeq H^1(X, \Z/n\Z)$
with the quadratic form $(\Cl(X)/n)^\vee\to \Z/n\Z$, $f\mapsto (f, f)_c$, of the class invariant pairing $(\cdot, \cdot )_c$.
This statement
was first shown in \cite{BCGKPT} by a different argument. This result of \cite{BCGKPT} inspired us to 
obtain the above theorem.

\item
Under the additional hypothesis that $\mu_{n^2}\subset F$, the pairing $(\cdot,\cdot )$ is symmetric and agrees
with the pairing defined in Section \ref{section2}. This follows from Lemma \ref{lem11} and its proof.
\end{enumerate}
\end{rem}

\begin{cor}
Assuming   $\mu_{n^2}\subset F$, the class invariant pairing $(\cdot ,\cdot )_c$ is symmetric.  
\end{cor}
\begin{proof}
This follows from Lemma \ref{lem11} and its proof and Theorem \ref{cithm}.
\end{proof}

\begin{proof}[Proof of Theorem \ref{cithm}]
Recall that Artin-Verdier duality \cite{Maz} gives isomorphisms
\begin{equation}\label{1}
H^1(X, \Z/n\Z)^\vee\simeq \Ext^{2}_X(\Z/n\Z, \Gm),\quad H^2(X, \Z/n\Z)^\vee\simeq \Ext^{1}_X(\Z/n\Z, \Gm).
\end{equation}
Applying $\Ext^i_X(-,\Gm)$ to $0\to \Z\xrightarrow{\times n} \Z\to \Z/n\Z\to 0$ gives an exact sequence
\begin{equation}\label{exact1}
0\to \Ext^0_X(\Z, \Gm)/n\xrightarrow{\partial} \Ext^1_X(\Z/n\Z, \Gm)\to \Ext^1_X(\Z, \Gm)[n]=\Cl(X)[n]\to 0,
\end{equation}
where the connecting $\partial$ is  given via the Yoneda product
\[
\Ext^0_X(\Z, \Gm)\times \Ext^1(\Z/n\Z, \Z)\to \Ext^1_X(\Z/n\Z, \Gm)
\]
with the class of $R(n)=(0\to \Z\xrightarrow{\times n} \Z\to \Z/n\Z\to 0)$. This combined with (\ref{1}) induces
a surjective homomorphism
\[
h: H^2(X,\Z/n\Z)^\vee\to \Cl(X)[n].
\]
Similarly, we have
\[
0\to \Ext^1_X(\Z, \Gm)/n\xrightarrow{\partial'} \Ext^2_X(\Z/n\Z, \Gm)\to \Ext^2_X(\Z, \Gm)[n]={\rm Br}(X)[n]=0,
\]
in other words, an isomorphism
\begin{equation*}\label{artin1}
\partial': \Cl(X)/n\xrightarrow{\simeq } \Ext^2_X(\Z/n\Z, \Gm).
\end{equation*}
The composition of $\partial'$ with the duality   
$\Ext^2_X(\Z/n\Z, \Gm)\simeq H^1(X,\Z/n\Z)^\vee$ is the dual ${\mathrm {Ar}}^\vee$ of the isomorphism
\[
{\mathrm {Ar}}: H^1(X,\Z/n\Z)\xrightarrow{\sim} (\Cl(X)/n)^\vee
\]
given by the Artin map of class field theory, i.e. ${\mathrm {Ar}}(x)$ is the Artin reciprocity map 
$\Cl(X)\to \Z/n\Z$ for the $\Z/n\Z$-torsor given by $x$ (see \cite[p. 539]{Maz}).

 Taking Yoneda product with the class
  \[
  [E(n)]=(0\to \Z/n\Z\to \Z/{n^2}\Z\to \Z/n\Z\to 0)
  \]
   in $\Ext^1_\Z(\Z/n\Z, \Z/n\Z)$
gives the Bockstein homomorphisms:
\begin{equation*}\label{B1}
\delta'': \Ext^{1}_X(\Z/n\Z, \Gm)\to \Ext^{2}_X(\Z/n\Z, \Gm),
\end{equation*}
\begin{equation*}\label{B2}
\delta': \Ext^{1}_X(\Z, \Z/n\Z)=H^1(X,\Z/n\Z)\to \Ext^{2}_X(\Z, \Z/n\Z)=H^2(X, \Z/n\Z).
\end{equation*}
Under the duality isomorphisms (\ref{1}), the   dual $\delta'^\vee$
is identified with the Bockstein $\delta''$. This easily follows 
from the fact that the Artin-Verdier duality pairings 
are also given via Yoneda products.

\begin{prop}\label{Bockprop}
The dual $\delta'^\vee: \rH^2(X, \Z/n\Z)^\vee\to \rH^1(X, \Z/n\Z)^\vee $ of the Bockstein homomorphism $\delta'$
is equal to the composition
\[
H^2(X, \Z/n\Z)^\vee\xrightarrow{\ h\ } \Cl(X)[n]\to \Cl(X)/n\xrightarrow{{\mathrm {Ar}}^\vee} H^1(X,\Z/n\Z)^\vee.
\]
where the map $\Cl(X)[n]\to \Cl(X)/n$ is induced by the identity on $\Cl(X)$.
\end{prop}

 \begin{proof}
Consider the composition $\delta'\circ \partial$ where $\partial$ is as in (\ref{exact1}).
The connecting $\partial$ is  given as Yoneda product
\[
\Ext^0_X(\Z, \Gm)\times \Ext^1(\Z/n\Z, \Z)\to \Ext^1_X(\Z/n\Z, \Gm)
\]
with the class of $R(n)=(0\to \Z\to \Z\to \Z/n\Z\to 0)$.
Hence the composition $\delta'' \circ \partial$ is given by
\[
(\beta' \circ\partial) (a)=a\cup [R(n)]\cup [E(n)].
\]
But $[R(n)]\cup [E(n)]=0$, since $\Ext^2_\Z(\Z/n\Z, \Z)=(0)$. Therefore,   $\delta'' $ 
factors through the quotient by the image of $\partial$:
\[
\delta'': \Ext^1_X(\Z/n\Z, \Gm)\to \Ext^1_X(\Z, \Gm)[n]=\Cl(X)[n]\xrightarrow{} \Ext^2_X(\Z/n\Z, \Gm).
\]
Combining this with the isomorphism $\partial'$ gives a factorization of $\delta''$ as
a composition
\[
\Ext^1_X(\Z/n\Z, \Gm)\to  \Cl(X)[n]\xrightarrow{\epsilon} \Cl(X)/n\xrightarrow{\sim } \Ext^2_X(\Z/n\Z, \Gm).
\]
Since duality identifies $\delta'^\vee$ with $\delta''$ it remains to see that, 
in the above, $\epsilon$ is induced by the identity map on $\Cl(X)$:

Write an element $y\in \Cl(X)[n]$ as the extension $1\to\Gm\to J'\to \Z\to 0$
coming from pulling back $x=(1\to\Gm\to J\to \Z/n\Z\to 0)\in \Ext^1_X(\Z/n\Z, \Gm)$ via $\Z\to \Z/n\Z$.
Then $\delta''(x)$ is the class of
\[
1\to \Gm\to J\to \Z/{n^2}\Z\to \Z/n\Z\to 0
\]
concatenating $x$ with $E(n)$. On the other hand, $y$ corresponds under $\Cl(X)/n\xrightarrow{\simeq } \Ext^2_X(\Z/n\Z, \Gm)$ to the extension 
\[
1\to \Gm\to J'\to \Z\to \Z/n\Z\to 0 
\]
obtained by concatenating $1\to\Gm\to J'\to \Z\to 0$ with $R(n): 0\to \Z\to \Z\to \Z/n\Z\to 0$.
Pushing out $ R(n): 0\to \Z\to \Z\to \Z/n\Z\to 0$ by $\Z\to \Z/n\Z$ gives
$ E(n): 0\to \Z/n\Z\to \Z/{n^2}\Z\to \Z/n\Z\to 0$ and so we have a commutative
diagram
\[
\xymatrix{
1 \ar[r]  & \Gm \ar[r] & J \ar[r] & \zmod {n^2} \ar[r] & \zmod n \ar[r]  & 0 \\
1 \ar[r] & \Gm \ar[u] \ar[r] & J' \ar[u] \ar[r] & \Z \ar[u] \ar[r] & \zmod n \ar[u] \ar[r] & 0
}
\]
which shows the statement. This concludes the proof of the Proposition.
\end{proof}

By the definition of the arithmetic linking pairing
\[
(\cdot , \cdot ) :\ H^1(X, \Z/n\Z)\times H^1(X, \Z/n\Z)\to \frac{1}{n}\Z/\Z=\Z/n\Z
\]
the corresponding homomorphism $D: H^1(X, \Z/n\Z)\to H^1(X, \Z/n\Z)^\vee$
(i.e. with $D(x)(x')=(x, x')$)
is given as the composition
\[
H^1(X, \Z/n\Z)\to H^2(X, \Z/n\Z)^\vee\xrightarrow{\delta'^\vee} H^1(X, \Z/n\Z)^\vee
\]
 of the homomorphism given by cup product and Artin-Verdier duality with the dual of the Bockstein. 
 By combining this with Proposition \ref{Bockprop} we see that $D$ is the composition  
 \begin{equation*}\label{result}
H^1(X,\Z/n\Z)\xrightarrow{\ } H^2(X,\Z/n\Z)^\vee\xrightarrow{\ h\ } \Cl(X))[n]\to \Cl(X)/n\xrightarrow{{\mathrm {Ar}}^\vee} H^1(X,\Z/n\Z)^\vee.
\end{equation*}

\begin{lem}
Suppose that the $\Z/n\Z$-torsor $x\in  H^1(X,\Z/n\Z)$ has generic fiber $ F(\xi^{1/n})/F$ where $\xi\in F^*$ is a Kummer generator. Then the fractional ideal of $F$ generated by $\xi$ is the $n$-th power
$
(\xi)= I^n
$
of a well-defined fractional ideal $I$ of $F$; the class $[I]=[I(x)]\in \Cl(X)$ only depends on $x$, is $n$-torsion,
and is equal to the image $\Psi(x)$ of the class invariant homomorphism.
The image of $x$ under the composition of the first two maps above
\[
 H^1(X,\Z/n\Z)\xrightarrow{\ }  H^2(X,\Z/n\Z)^\vee\xrightarrow{\ h\ } \Cl(X)[n]
\]
is $\Psi(x)=[I(x)]$. 
\end{lem}

\begin{proof}  The first part of the statement is standard. In fact, we have $i: \calL\otimes_{\O_F}F\simeq F\cdot \xi^{1/n}\simeq F$ and, by definition, $I=i(\calL)$ and so $\Psi(x)=[I(x)]$.

The rest of the statement of the lemma follows from Artin-Verdier duality, the computation of  the group $H^2(X,\Z/n\Z)^\vee\simeq \Ext^1_X(\Z/n\Z, \Gm)$, and of the local duality pairings via Hilbert symbols,
in \cite{Maz}, see p. 540--541 and p. 550--551. A more detailed statement 
appears in \cite{BCGKPT}.
\end{proof}

It now follows that $D:   H^1(X, \Z/n\Z)\to H^1(X, \Z/n\Z)^\vee$
is the map
\[
x\mapsto (x'\mapsto {\rm {Ar}}(x')([I(x)]))\in {\Hom}(H^1(X, \Z/n\Z), \Z/n\Z),
\]
where ${\rm {Ar}}({x'}): \Cl(X)\to \Z/n\Z$ is the Artin (reciprocity) homomorphism
associated to the $\Z/n\Z$-torsor given by $x'$. 
The statement of the theorem follows.
\end{proof}

\begin{rem}$~$
\begin{enumerate}[a)]
\item
It follows from the above description of the map 
\[
H^1(X,\Z/n\Z)\xrightarrow{D} H^1(X,\Z/n\Z)^\vee\simeq \Ext^2_X(\Z/n\Z, \Gm)\simeq \Cl(X)/n
\]
that the group of $n$-homologically trivial ideal classes in $\Cl(X)/n$ coincides with the 
image of the class invariant homomorphism $\Psi$ in $\Cl(X)/n$. 
In the theory of Galois module structure, ideal classes which are in the image of the class invariant homomorphism are often called `realisable'. 

\item
Assuming $\mu_{n^2}\subset F$, the class invariant pairing can be viewed as a canonical symmetric tensor
\[
c(F, n)\in {\rm TS}^2_{\Z/n\Z}(\Cl(F)/n):=(\Cl(F)/n\otimes \Cl(F)/n)^{S_2}.
\]
It would be interesting to study this tensor and its variation in families of number fields.
\end{enumerate}
\end{rem}

\section*{Acknowledgments}
The authors are very grateful to Kai Behrend, Frauke Bleher, Ted Chinburg,  Tudor Dimofte, Ralph Greenberg, Andre Henriques, Mahesh Kakde, Effie Kalfagianni, Mikhail Kapranov, Kobi Kremintzer,  Graeme Segal, Romyar Sharifi, Martin Taylor, and Roland van der Veen for conversations and encouragement.
\ms

M.K. was supported in part by the EPSRC grant EP/M024830/1. 

G.P. was  supported in part by NSF Grant No. DMS-1360733.

J.P. was supported  in part by the Samsung Science \& Technology Foundation (SSTF-BA1502- 03).

H.Y. was supported in part by the Institute for Basic Sciences grant IBS-R003-D1.

\bibliographystyle{annotation}

\begin{thebibliography}{99}
\bibitem{AK} Arnold, V.I.; Khesin, B. A.,
\emph{Topological methods in hydrodynamics}, 
Applied Mathematical Sciences, 125, Springer-Verlag, New York (1998).

\bibitem{BCGKPT} Bleher, F., Chinburg, T., Greenberg, R., Kakde, M., Pappas, G., Taylor, M., \emph{Unramified arithmetic Chern-Simons invariants}, preprint, arXiv:1705.07110.

\bibitem{CF} Cassels, J.W.S., Fr\"ohlich, A., \emph{Algebraic number theory}, Proceedings of an instructional conference organized by the LMS, Academic Press Inc (London) LTD (1967).

\bibitem{CKKPY} Chung, H.-J., Kim, D., Kim, M., Park, J., Yoo, H., \emph{Arithmetic Chern-Simons theory II}, preprint, arXiv:1609.03012.


\bibitem{IR} Illusie, L., Raynaud, M., \emph{Les suites spectrales associ\`ees au complexe de de Rham-Witt}, Publ. IHES, No. 57, 73--212 (1983).
\bibitem{Kim} Kim, M., \emph{Arithmetic Chern-Simons theory I}, preprint, arxiv: 1510.05818.
\bibitem{Maz} Mazur, B., \emph{Notes on \'etale cohomology of number fields},  Ann. Sci. \'Ecole Norm. Sup. (4) 6, 521--552 (1973).

\bibitem{Mil} Milne, J., \emph{\'Etale cohomology}, Princeton Mathematics Series 33, Princeton University Press (1980).

\bibitem{Mil2} Milne, J., \emph{Lectures on \'Etale cohomology} (Ver. 2.21), available at \url{http://jmilne.org/math/CourseNotes/lec.html}.

\bibitem{Mor}Morishita, M., \emph{Knots and primes. An introduction to arithmetic topology}, Universitext, Springer, London (2012).

\bibitem{Ner}  Neretin, Yu. A., \emph{Lectures on Gaussian integral operators and classical groups}, EMS Series of Lectures in Mathematics. European Mathematical Society (EMS), Zuerich (2011). 

\bibitem{Nek} Nekov\'ar, J.,
\emph{Selmer complexes}. 
Ast\'erisque No. 310 (2006).

\bibitem{Pol} Polyakov, A.M., \emph{Modern Physics Letters A}, Volume 3, Issue 03, 325--328 (1988). 

\bibitem{Sch} Schwartz, A.S., \emph{The partition function of degenerate quadratic functional and Ray-Singer invariants}, Lett. Math. Phys. \textbf{2}, no. 3, 247--252 (1977/78).

\bibitem{Tay} Taylor, M. J., \emph{The Galois module structure of certain arithmetic principal homogeneous spaces},
J. Algebra 153, no. 1, 203--214 (1992). 

\bibitem{Wat} Waterhouse, W. C.,
\emph{Principal homogeneous spaces and group scheme extensions},  
Trans. Amer. Math. Soc. 153, 181--189 (1971). 

\end{thebibliography}

\end{document}